\documentclass[11pt,a4paper]{article}

\usepackage[T1]{fontenc}
\usepackage[utf8]{inputenc}
\usepackage{lmodern}
\usepackage[margin=30mm]{geometry}
\usepackage{amsmath,amssymb,amsthm,mathtools}
\usepackage{microtype}
\usepackage{enumitem}
\usepackage{xcolor}
\usepackage{aliascnt}
\usepackage{hyperref}
\hypersetup{colorlinks=true,linkcolor=blue!45!black,citecolor=blue!45!black,
  urlcolor=blue!45!black,pdftitle={Lower curvature bounds on metric surfaces via the Bishop--Gromov inequality},
  pdfauthor={Alexander Lytchak and Stephan Stadler}}
\usepackage[nameinlink,noabbrev,capitalise]{cleveref}
\setlist[enumerate]{label=\textup{(\roman*)},leftmargin=*,itemsep=3pt}
\numberwithin{equation}{section}
\newtheorem{theorem}{Theorem}[section]
\newaliascnt{proposition}{theorem}
\newtheorem{proposition}[proposition]{Proposition}
\aliascntresetthe{proposition}
\newaliascnt{lemma}{theorem}
\newtheorem{lemma}[lemma]{Lemma}
\aliascntresetthe{lemma}
\newaliascnt{corollary}{theorem}

\aliascntresetthe{corollary}
\theoremstyle{definition}
\newaliascnt{definition}{theorem}

\aliascntresetthe{definition}
\newaliascnt{remark}{theorem}
\newtheorem{remark}[remark]{Remark}
\aliascntresetthe{remark}

\newcommand{\R}{\mathbb R}

\newcommand{\D}{\mathbb D}
\newcommand{\HH}{\mathcal H}
\newcommand{\LL}{\mathcal L}
\newcommand{\eps}{\varepsilon}
\newcommand{\length}{\operatorname{length}}

\newcommand{\apmd}{\operatorname{ap\,md}}
\newcommand{\essliminf}{\operatorname*{ess\,liminf}}
\newcommand{\loc}{\mathrm{loc}}
\newcommand{\CBB}{\operatorname{CBB}}
\newcommand{\vk}{v_\kappa}
\newcommand{\Vk}{V_\kappa}
\newcommand{\dd}{\,d}

\title{Lower curvature bounds on metric surfaces\newline via the Bishop--Gromov inequality}
\author{Alexander Lytchak \and Stephan Stadler}
\date{}

\begin{document}
\maketitle
\begin{abstract}
We characterize lower curvature bounds on metric surfaces   by the Bishop--Gromov inequality for metric balls.
\end{abstract}

\section{Introduction}
For a surface $M$ with a smooth Riemannian metric and any point $x\in M$
with Gaussian curvature $\kappa=\kappa(x)$, the area of small balls
$B_r(x)$ of radius $r$ around $x$ can be expressed as
\[
 \mathcal H^2(B_r(x))
 =\pi r^2-\frac{\pi\kappa}{12}r^4+O(r^6)
 =v_\kappa(r)+O(r^6),
\]
where $v_\kappa(r)$ denotes the area of a ball of radius $r$ in the
complete simply connected surface of constant curvature $\kappa$;
see \cite[Equation~(1.2)]{GrayVanhecke}.

Thus, if for a fixed $\kappa\in\R$ and every $x\in M$ the condition
\begin{equation}\label{eq:main}
 \mathcal H^2(B_r(x))\leq v_\kappa(r)
\end{equation}
holds for all sufficiently small $r$, then the curvature of $M$ is bounded below by
$\kappa$ on all of $M$. On the other hand, if the curvature of $M$
is bounded below by $\kappa$, then \eqref{eq:main} is true for all
$x$ and all $r$ such that the closure of $B_{2r}(x)$ is compact in
$M$, due to the theorem of Bishop--Gromov; see
\cite[Theorem~10.2]{BGP}.

We extend this equivalence between upper bounds on the areas of balls
and lower curvature bounds (in the sense of Alexandrov) to metric
surfaces without any smoothness assumption. Such a statement requires
excluding Finsler structures.
Indeed, on any non-Euclidean normed plane, which has no curvature
bounds in the sense of Alexandrov, the Hausdorff area of any ball
of radius $r$ is $\pi r^2$; see \cite[Section~3.1]{Alvarez}.
Finsler structures may be excluded by changing the reference measure
$\mathcal H^2$ to an appropriate variant (e.g., the inscribed
Riemannian area introduced by Ivanov \cite[Example~4.4]{Ivanov} or a Hilbert-type
area introduced by Gromov \cite[Section~8]{GromovHilbert}). We retain
Hausdorff area and exclude Finsler structures by 
a direct
infinitesimal Euclidean assumption (ET) introduced in
\cite[Definition~11.1]{LWArea}; see also \cref{rm:rm}.

\begin{theorem}\label{thm:main}
Let $X$ be a length space homeomorphic to a surface without boundary
and let $\kappa\in\R$. Then $X$ has curvature locally bounded below
by $\kappa$ if and only if it satisfies \textup{(ET)} and, for every
compact $K\subset X$, there exists $r_K>0$ such that
\begin{equation}\label{eq:main-ball}
 \HH_X^2(B_r^X(x))\leq v_\kappa(r)
 \qquad\text{for all }x\in K\text{ and }0<r<r_K.
\end{equation}
\end{theorem}

Somewhat surprisingly, the length-metric assumption can largely be  dispensed with, as explained in
Appendix~\ref{app:intrinsic}. If $\kappa\geq0$ and a metric surface
$(X,d)$ satisfies property~(ET) and \eqref{eq:main-ball}, then its
induced length metric $d_{\mathrm{length}}$, considered on each
connected component, has curvature locally bounded below by $\kappa$.
For $\kappa<0$, the induced extended length metric may instead have
two finite-distance components: a point and a hyperbolic cusp of
constant curvature $\kappa$, at infinite distance from each other.
This cusp phenomenon also accounts for some of the difficulties in
the proof of \cref{thm:main}; see \cref{sec:cusps}.

\begin{remark}\label{rm:rm}
For length surfaces $X$ with locally finite Hausdorff $2$-measure,
property~(ET), recalled in \cref{sec:preliminaries}, is equivalent
to each of the following conditions:
\begin{itemize}
\item $(X,d,\HH_X^2)$ is infinitesimally Hilbertian, that is,
$W^{1,2}(X,d,\HH_X^2)$ is a Hilbert space. For this notion, see
\cite[Definition~4.19]{Gigli} and \cite[Corollary~6.7]{Soultanis}.
\item Every $2$-rectifiable subset of $X$ has a Euclidean approximate
tangent space in the sense of \cite[Section~5]{AK} at
$\HH_X^2$-almost every point.
\end{itemize}
These equivalences are not used in the proof of \cref{thm:main},
and we omit their proofs.
\end{remark}

The  proof follows a strategy similar to those used in \cite{LS} and \cite{LWcharacterization}. We first find
a conformal parametrization of a small neighborhood in $X$, relying on
\cite{Rajala,LWCanonical,MW,NR}, and then determine the analytic properties of
the conformal factor $f$ imposed by \eqref{eq:main-ball}. A delicate
analytic argument turns this condition into the inequality
\begin{equation} \label{eq:intro}
 \Delta\log f+\kappa f^2\leq0
\end{equation}
in the sense of distributions. Subsequently, using the theory of
surfaces of bounded integral curvature developed by Alexandrov, Zalgaller
and Reshetnyak \cite{AZ,Reshetnyak}, we verify that the conformal factor
$f$ determines the metric of $X$ everywhere, rather than merely almost
everywhere, and that this metric has curvature locally bounded below.

Unlike in the related paper \cite{LS}, here one needs the a priori topological assumption that $X$ is a surface without boundary. Indeed, a triangle in the plane with one side
slightly bent inward, equipped with its intrinsic length metric,
satisfies the area assumptions of Theorem~\ref{thm:main} for $\kappa=0$,
but is not locally $\CBB(0)$.

After the preliminaries, \cref{sec:conformal-charts} constructs local
conformal parametrizations of the space $X$.  Section
\ref{sec:radial} and
\cref{sec:representative}, the technical heart of the paper, investigate the conformal factor and  establish the analytic curvature inequality \eqref{eq:intro}. \cref{sec:identification}  upgrades the analytic structure to a metric one, while \cref{sec:cusps} completes the proof by extending across possible cusps in negative curvature.
Appendix~\ref{app:intrinsic} discusses non-intrinsic metrics.

\subsection*{Acknowledgments}
This paper arose in close collaboration with ChatGPT, an AI assistant
developed by OpenAI. The collaboration involved exploring proof
strategies, refining arguments, locating references, and drafting
and revising the manuscript. The authors take full responsibility
for the mathematical content and the final text.
The first author was supported by the Deutsche Forschungsgemeinschaft
(DFG, German Research Foundation) under Germany's Excellence
Strategy -- EXC-2047/1 -- 390685813.

\section{Preliminaries}\label{sec:preliminaries}

%\subsection{Metric surfaces}
We assume familiarity with the basic concepts of metric geometry and
refer the reader to \cite{BBI,AKP}. Here we fix notation and terminology.

In a metric space, we denote by $B_r(x)$ the open ball of radius $r$
centered at $x$. Hausdorff measure is normalized so that $\HH^2$ on the
Euclidean plane agrees with Lebesgue measure $\LL^2$. A \emph{metric
surface} is a metric space homeomorphic to a two-dimensional manifold
without boundary. A length space is a metric space in which distances
are infima of lengths of curves. A geodesic space is one in which these
infima are attained. A locally compact length space is locally geodesic:
every point has a neighborhood any two points of which can be joined
by a minimizing curve in the ambient space; see \cite[Section~2.5]{BBI}.

We write $\CBB(\kappa)$ for the Alexandrov lower curvature bound
$\kappa$, and locally $\CBB(\kappa)$ for the corresponding local
condition. We refer to \cite{AKP} for the definitions and theory.

%\subsection{Almost every curve and Sobolev maps}\label{subsec:sobolev}
All metric surfaces will be considered with $\HH^2$ as the reference measure.
A property holds for \emph{$2$-almost every rectifiable curve}
in a metric surface $Y$ if there exists a Borel function
$h:Y\to[0,\infty]$ with $h\in L^2(Y)$ such that
\[
 \int_\gamma h\dd s=\infty
\]
for every rectifiable curve $\gamma$ for which the property fails.
The family of such exceptional curves is called \emph{$2$-exceptional}.
Countable unions of $2$-exceptional families are
$2$-exceptional.
For this and other properties, see \cite{HeinonenKST} and
\cite[Subsection~3.1]{LWIntrinsic}.

We denote by $\D$ the open unit disc in $\R^2$. For a planar domain
$D$ and a complete metric space $Z$, we write $W^{1,2}(D,Z)$ for the
space of metric-valued Sobolev maps; see \cite{HeinonenKST},
\cite[Section~3]{LWArea} and \cite[Subsection~3.2]{LWIntrinsic}.
In assertions about almost every curve, a Newton--Sobolev
representative is understood. A continuous Sobolev map is already such
a representative; see \cite[Subsection~3.2]{LWIntrinsic}.

Every $u\in W^{1,2}(D,Z)$ has an approximate metric differential
$\apmd u_z$, a seminorm on $\R^2$, at almost every $z\in D$;
see \cite[Proposition~4.3]{LWArea}. There is a decomposition
\[
 D=N\sqcup\bigcup_{i=1}^{\infty}K_i,
 \qquad \LL^2(N)=0,
\]
where the $K_i$ are pairwise disjoint compact sets and each restriction
$u|_{K_i}$ is Lipschitz. At almost every $z\in K_i$, its metric
differential in the sense of \cite{Kirchheim} agrees with $\apmd u_z$;
see \cite[Subsection~3.3]{LWIntrinsic}.

For a possibly noncomplete metric space $Z$ with completion $\widehat Z$,
we define $W^{1,2}(D,Z)$ as the maps in $W^{1,2}(D,\widehat Z)$
with values in $Z$ almost everywhere. Approximate metric differentials,
energies and Newton--Sobolev representatives are understood in
$\widehat Z$.

A  metric space $Z$ has property~(ET) if, for every
$u\in W^{1,2}(\D,Z)$, the approximate metric differential $\apmd u_z$
is a Euclidean seminorm for almost every $z\in\D$.

A map $u\in W^{1,2}(D,Z)$ is called \emph{weakly conformal} if
$\apmd u_z=f(z)|\cdot|$ almost everywhere, for a nonnegative function
$f\in L^2(D)$.

\section{Conformal parametrization}\label{sec:conformal-charts}
The following result is a consequence of metric uniformization theorems
proved in various settings in
\cite{Bonk-Kleiner,Rajala,LWCanonical,MW,NR}.

\begin{proposition}\label{prop:charts}
Let $X$ be a length metric surface satisfying property~\textup{(ET)}
and \eqref{eq:main-ball}. For every $x\in X$ there is a topological embedding $\Phi:\D\to X$ with $\Phi(0)=x$ such that
\begin{enumerate}
\item $\Phi\in W^{1,2}(\D,X)$ and
$\apmd\Phi_z=f(z)|\cdot|$ almost everywhere, where
$f\in L^2(\D)$ is positive almost everywhere;
\item for $2$-almost every rectifiable curve $\gamma$ in $\D$,
\[
 \length_X(\Phi\circ\gamma)=\int_\gamma f\dd s,
\]
and for every Borel set $E\subset\D$,
\[
 \int_E f^2\dd z\leq\HH_X^2(\Phi(E)).
\]
\end{enumerate}
\end{proposition}

\begin{proof}
By \cite[Lemma~5.1]{CreutzRomney}, we find an arbitrarily small,
compact neighborhood $\Omega$ of $x$ homeomorphic to $\overline \D$ with
rectifiable boundary $\partial\Omega$. Then $\Omega$
with its intrinsic metric is a compact geodesic disc $Y$ with the same
Hausdorff area and curve lengths as in $X$; see \cite[Lemma~2.1]{LWCanonical}. By compactness, the bound
\eqref{eq:main-ball} gives
$\HH_Y^2(B_r^Y(y))\leq Cr^2$ for some $C>0$, all $y\in Y$ and all $r>0$.

By \cite[Theorems~1.3, 1.4, 2.3]{MW}, there is a continuous map
$u:\overline\D\to Y$ whose restriction to $\partial\D$ is a weakly
monotone parametrization of $\partial Y$ and whose restriction to $\D$ lies in  $W^{1,2}(\D,Y)$ and minimizes the so-called \emph{Reshetnyak energy}
among all fillings of $\partial Y$.

Since the metrics agree locally in the interior of $\Omega$ and
$\HH_Y^2(\partial Y)=0$, $Y$ inherits property~(ET) from $X$, as in
\cite[Section~4.5]{LS}. Hence \cite[Theorem~11.4]{LWArea} implies that
$u$ is weakly conformal.

The map $u$ is surjective and monotone, i.e., the preimage of every
point of $Y$ is connected, by
\cite[Theorem~1.2 and Proposition~2.9]{LWCanonical}. The quadratic area bound implies that $u$ is a homeomorphism, due to 
\cite[Theorem~3.6]{LWCanonical}.

The inverse map $u^{-1}:Y\to\overline\D$ has Lusin's property~(N),
meaning that it maps $\HH_Y^2$-null sets to $\LL^2$-null sets.
Indeed, the quadratic area bound and \cite[Theorem~1.6]{Rajala},
together with \cite[Proposition~3.5]{LWCanonical} and
\cite[Proposition~3.3]{MW}, applied on compact subdiscs of $\D$,
imply that $u|_{\D}:\D\to\operatorname{int}Y$ is locally
geometrically quasiconformal.
By \cite[Theorem~6.20]{NtalampekosSurvey}, $u^{-1}$ therefore has
property~(N) on $\operatorname{int}Y$. Since
$u^{-1}(\partial Y)=\partial\D$ has Lebesgue measure zero,
property~(N) holds on all of $Y$.

Precomposing with
a M\"obius transformation of $\overline\D$, we may assume that $u(0)=x$.

Let $\iota:Y\to X$ be the canonical $1$-Lipschitz identification and set
$\Phi=\iota\circ u|_{\D}$. Then $\Phi$ is a topological embedding in
$W^{1,2}(\D,X)$ with $\Phi(0)=x$. Since $\iota$ is locally isometric
in the interior, $\apmd\Phi_z=f(z)|\cdot|$ almost everywhere for a
nonnegative $f\in L^2(\D)$. The asserted length identity follows from
\cite[Lemma~3.1]{LWIntrinsic}.

Let $A\subset\D$ be a countable union of compact pieces on which $\Phi$ is Lipschitz with $\mathcal L^2 (\D\setminus A)=0$.
Injectivity and the metric area formula
\cite{Kirchheim} imply 
\begin{equation} \label{eq:E}
\int_E f^2\dd z
=\HH_X^2\bigl(\Phi(E\cap A)\bigr)
\leq\HH_X^2\bigl(\Phi(E)\bigr),
\end{equation}
 for every Borel set $E\subset\D$.

It remains to verify that  $f$ is positive almost
everywhere. Otherwise, we could find a compact subset $E$ of positive measure in $f^{-1} (0) \cap A$, where $A\subset\D$ is as above. The metric area formula \eqref{eq:E}  would imply that $\Phi (E)=u(E)$ had $\mathcal H^2$-measure $0$. Hence, the inverse map $u^{-1}$ would not satisfy Lusin's property (N). This contradiction finishes the proof.  
\end{proof}

Let $\Phi,f$ be as in \cref{prop:charts}. We call such maps $\Phi$
\emph{charts} and the associated functions $f$ \emph{conformal
	factors}. If $h$ is a conformal diffeomorphism of parameter domains,
then $\Phi\circ h$ is again a chart, with conformal factor
$(f\circ h)|h'|$. Since $f>0$ almost everywhere, we can choose a Borel
representative of $f$ that is positive everywhere. By \cite[Proposition~4.10]{LWArea}, applied to
parallel-line coordinates, and Fubini's theorem, for almost every
$x\in\D$ and almost every $\theta\in S^1$, the curve
$t\mapsto\Phi(x+t\theta)$ is absolutely continuous and has finite
length on every interval $[0,T]$ with $[x,x+T\theta]\subset\D$,
and satisfies the length identity in \cref{prop:charts} on every
subsegment. Also, by polar integration, $f$ is finite
almost everywhere along almost every such ray.

We call such an $x$ a \emph{good center}, and choose a Borel
full-measure set of \emph{good directions} at $x$ on which all these
properties hold.

\section{Conformal factor: uniform positivity}\label{sec:radial}
We fix some notation and analyze the conformal factor $f$ from
\cref{prop:charts} more closely in this and the next section.
All arguments are local.
We fix a compact set
$K\Subset \D$ and choose $R>0$ such that the $3R$-neighborhood
of $K$ is contained in $\D$. If necessary, enlarge $K$ slightly, so that
the estimates obtained on $K$ hold on a neighborhood of each of its
points. All constants below are uniform in $x\in K$.

Compactness and injectivity of $\Phi$ imply
\begin{equation}\label{eq:separation}
 \inf_{x\in K,\,\theta\in S^1}
 d_X\bigl(\Phi(x),\Phi(x+R\theta)\bigr)>0.
\end{equation}
Choose $r_0>0$ such that 
$2r_0$ is smaller than this infimum and smaller than $R$, such
that the ball-area assumption applies at every $\Phi(x)$, $x\in K$,
for $0<r<2r_0$.

For a good center $x\in K$ and a good direction $\theta\in S^1$, set
\[
 Q_{x,\theta}(t)=\int_0^t f(x+s\theta)\dd s,
 \qquad0\leq t\leq R.
\]
This is the length of the curve $s\mapsto\Phi(x+s\theta)$,
$0\leq s\leq t$. The function $Q_{x,\theta}$ is absolutely
continuous and strictly increasing. Since $Q_{x,\theta}(R)\geq 2r_0$,
define
\begin{equation}\label{eq:inverse}
 \tau_{x,\theta}(r)=Q_{x,\theta}^{-1}(r),
 \qquad0<r<2r_0.
\end{equation}
For each good center $x$, the functions $Q_{x,\theta}(t)$ and
$\tau_{x,\theta}(r)$ are Borel measurable in $(\theta,t)$ and
$(\theta,r)$, respectively. Set $\tau_{x,\theta}(r)=R$ on the
negligible set of bad directions.

For a good center $x$ define
\begin{equation}\label{eq:radial-area}
 A_x(r)=\int_{S^1}\int_0^{\tau_{x,\theta}(r)}
                   f(x+t\theta)^2t\dd t\dd\theta.
\end{equation}
The  radial star
$\{x+t\theta:0\leq t<\tau_{x,\theta}(r),\ \theta\text{ good}\}$
is mapped by $\Phi$ into $B_r^X(\Phi(x))$, since its initial radial
segments have length less than $r$. Thus the area inequality in
\cref{prop:charts} and the ball-area bound give
\begin{equation}\label{eq:A-upper}
 A_x(r)\leq\vk(r).
\end{equation}

Finally, choose $C\geq0$, with $C=0$ if $\kappa\geq0$, such that
\begin{equation}\label{eq:coarse-area-bound}
 \vk(t)\leq\pi t^2(1+Ct^2),\qquad0<t\leq r_0.
\end{equation}

The arguments in this and the next section exploit the relation between
radial length and area. To explain the mechanism, consider  a good center. Along each ray, the accumulated length
\(Q_\theta\) satisfies \(Q_\theta'=f\), and
\[
t f(t\theta)^2+\frac{Q_\theta(t)^2}{t}
-2Q_\theta(t)f(t\theta)
=\frac{\bigl(tf(t\theta)-Q_\theta(t)\bigr)^2}{t}\geq0.
\]
Equality holds precisely when \(Q_\theta(t)/t\) is constant, that is,
when the conformal factor is constant along the ray. Integrating up to
\(Q_\theta=r\) and then over all directions gives
\[
A_x(r)+
\int_{S^1}\int_0^{\tau_{x,\theta}(r)}
\frac{Q_{x,\theta}(t)^2}{t}\,dt\,d\theta
\geq 2\pi r^2.
\]
Thus an upper bound on the area of the radial star gives a lower bound
on the second integral, with any variation of the conformal factor
along the rays contributing with the favorable sign.

We use this observation in two ways. First, applying the same square
inequality to a truncated logarithm of \(Q_\theta\), we compare its
radial energy with the area bound. At a Lebesgue center,
\(Q_\theta(t)/t\) approaches \(f(x)\) in angular average. If \(f(x)\)
were too small, reaching a fixed positive radial length within the
coordinate neighborhood would require more energy than the area
bound permits. This yields a locally uniform positive lower bound
for \(f\).

After normalizing \(f\geq1\), we introduce radial averages \(H_x(r)\)
which equal \(\log f\) for constant conformal factors. Their derivative
contains exactly the second integral above. The same inequality then
gives monotonicity after a curvature-dependent correction. This
provides a lower semicontinuous representative of \(\log f\) and
allows us to recover the sharp curvature inequality by comparison
with smooth lower tests.

The proof of the following result relies only on   \eqref{eq:A-upper}.

\begin{proposition}\label{prop:positivity}
Let $\Phi,f$ be as in \cref{prop:charts}. Then
$f^{-1}\in L^\infty_{\loc}(\D)$.
\end{proposition}

\begin{proof}
Fix a good center $x\in K$ which is also a Lebesgue point of $f$,
and put $a=f(x)\in(0,\infty)$. We will verify the explicit lower bound
\begin{equation}\label{eq:uniform-lower-bound}
 a\geq\frac{r_0}{R}\exp\bigl(-\tfrac12-Cr_0^2\bigr).
\end{equation}
By translation, we may assume $x=0$. Write
$Q_\theta=Q_{0,\theta}$ and $A=A_0$. We claim
\begin{equation}\label{eq:radial-diff}
 \frac{Q_\theta(t)}t\xrightarrow[t\to0]{}a
 \quad\text{in }L^1(S^1).
\end{equation}
Indeed, $F(t)=\int_{B_t(0)}|f-a|\dd z=o(t^2)$, and polar
integration and Tonelli's theorem give
\begin{align*}
 \int_{S^1}\left|\frac{Q_\theta(t)}t-a\right|\dd\theta
 &\leq\frac1t\int_{B_t(0)}\frac{|f(z)-a|}{|z|}\dd z\\
=\frac{1}{t}\int_0^t\frac{F'(v)}{v}\dd v &=\frac{F(t)}{t^2}+\frac1t\int_0^t\frac{F(v)}{v^2}\dd v
 \xrightarrow[t\to0]{}0.
\end{align*}

Choose $0<\eps<\min\{R,r_0/a\}$, and put $s=a\eps<r_0$ and
$L=\log(r_0/s)$. Truncating the shifted logarithm
$\log(Q_\theta/s)$, we define $u_\theta:[\eps,R]\to[0,L]$ by
\[
 u_\theta(t)=\min\bigl\{L,\max\{0,\log(Q_\theta(t)/s)\}\bigr\}.
\]
Along every good ray this function is absolutely continuous, with
$u_\theta(R)=L$ and
\[
 u_\theta'(t)=\frac{f(t\theta)}{Q_\theta(t)}
                 \mathbf 1_{\{s<Q_\theta(t)<r_0\}}
 \quad\text{almost everywhere}.
\]
Using $\log v\leq v-1$ for $v>0$ and \eqref{eq:radial-diff},
we obtain
\begin{equation}\label{eq:cutoff-inner-error}
 e_\eps:=\int_{S^1}u_\theta(\eps)\dd\theta
 \leq\frac1a\int_{S^1}
       \left|\frac{Q_\theta(\eps)}\eps-a\right|\dd\theta
 \xrightarrow[\eps\to0]{}0.
\end{equation}

Set $u(t\theta)=u_\theta(t)$ on $B_R(0)\setminus B_\eps(0)$ and
consider its radial energy
\[
 E_\eps=\int_{B_R(0)\setminus B_\eps(0)}|\partial_r u(z)|^2\dd z
       =\int_{S^1}\int_\eps^R t\,u_\theta'(t)^2\dd t\dd\theta,
\]
where $\partial_r$ denotes the radial derivative. Write
$Q(t\theta)=Q_\theta(t)$ on $B_R(0)$. The identity
\[
 \frac1{Q(z)^2}=\frac1{r_0^2}
       +2\int_{Q(z)}^{r_0}\frac{\dd q}{q^3}
 \qquad(s<Q(z)<r_0),
\]
the definition of the function $A$, and Tonelli's theorem give
\begin{align*}
 E_\eps
 &\leq\int_{\{z\in B_R(0):\,s<Q(z)<r_0\}}
                    \frac{f(z)^2}{Q(z)^2}\dd z\\
 &=\frac{A(r_0)-A(s)}{r_0^2}
       +2\int_s^{r_0}\frac1{q^3}
         \left(\int_{\{z\in B_R(0):\,s<Q(z)<q\}}
                    f(z)^2\dd z\right)\dd q\\
 &=\frac{A(r_0)-A(s)}{r_0^2}
       +2\int_s^{r_0}\frac{A(q)-A(s)}{q^3}\dd q\\
 &=\frac{A(r_0)}{r_0^2}-\frac{A(s)}{s^2}
                    +2\int_s^{r_0}\frac{A(q)}{q^3}\dd q\\
 &\leq\pi+2\pi L+2\pi Cr_0^2,
\end{align*}
where the last inequality uses \eqref{eq:A-upper} and
\eqref{eq:coarse-area-bound}.
On the other hand, integrating the elementary inequality
\[
2u_\theta'(t)\leq t\,u_\theta'(t)^2+1/t\,,
\] 
 and using absolute
continuity along the good rays, we obtain
\begin{align*}
 E_\eps
 &\geq\int_\eps^R\int_{S^1}
       \left(2u_\theta'(t)-\frac1t\right)\dd\theta\dd t\\
 &=2\int_{S^1}u_\theta(R)\dd\theta
       -2\int_{S^1}u_\theta(\eps)\dd\theta
       -2\pi\log\frac R\eps.
\end{align*}
Since $u_\theta(R)=L$ for almost every $\theta$, the first term is
$4\pi L$. The second term is $-2e_\eps$ by
\eqref{eq:cutoff-inner-error}. Combining this lower bound with the
preceding upper bound gives
\[
 \pi+2\pi L+2\pi Cr_0^2\geq E_\eps
 \geq4\pi L-2e_\eps-2\pi\log\frac R\eps.
\]
Rearranging, we obtain
\[
 \frac12+Cr_0^2+\frac{e_\eps}{\pi}
 \geq L-\log\frac R\eps
 =\log\frac{r_0}{a\eps}-\log\frac R\eps
 =\log\frac{r_0}{aR}.
\]
Letting $\eps\to0$ and using \eqref{eq:cutoff-inner-error}, we infer
\[
 \log\frac{r_0}{aR}\leq\frac12+Cr_0^2.
\]
This finishes the proof of  \eqref{eq:uniform-lower-bound} and of the proposition.
\end{proof}

For the subsequent local arguments choose a disc
$D_0=B_\rho(z_0)\Subset\D$. By \cref{prop:positivity}, there is
$c>0$ such that $f\geq c$ almost everywhere on $D_0$. After modifying
$f$ on a null set, this holds everywhere. After restricting $\Phi$ to
$D_0$ and precomposing with a homothety of $\R^2$, we obtain a chart
whose conformal factor is at least $1$ everywhere. We denote the new
domain by $D$ and retain the notation $\Phi,f$ for the chart and its
conformal factor.

From now on we will stick to this chart.
We  put
\[
 W=\log f\geq0\quad\text{on }D.
\]
Since $0\leq\log f\leq f^2$, we have $W\in L^1(D)$.

\section{The curvature inequality}
\label{sec:representative}

We continue with the normalized chart $\Phi:D\to U=\Phi(D)$ and
$W=\log f$ as  at the end of \cref{sec:radial}. The aim of this
section is to prove the distributional curvature inequality.

\begin{proposition}\label{prop:analytic}
For  $W=\log f$, the distribution $\Delta W$ is a signed Radon measure and 
\begin{equation}\label{eq:target}
 \Delta W +\kappa e^{2W} \leq 0\,.
\end{equation}
\end{proposition}

\subsection{From lower tests to distributions}
\label{subsec:lower-test-distribution}

We record a general analytic statement, independently
of the conformal chart. A $C^2$ function $\varphi$ is a \emph{lower test}
for a function  $u$ at a  point $x$ if $\varphi\leq u$ near $x$ and
$\varphi(x)=u(x)$.

\begin{lemma}[Lower tests imply the distributional inequality]
\label{lem:distribution}
Let $\Omega\subset\R^2$ be open and let
$u:\Omega\to\R\cup\{+\infty\}$ be lower semicontinuous and locally
bounded below, with $u,e^{2u}\in L^1_{\loc}(\Omega)$. Suppose that
for every $x\in \Omega$ with $u(x) <\infty$,  every lower test $\varphi$ at  $x$ satisfies
\begin{equation}\label{eq:abstract-lower-test}
 \Delta\varphi(x)+\kappa e^{2\varphi(x)}\leq0.
\end{equation}
Then $\Delta u+\kappa e^{2u}\leq0$ in the sense of distributions.
\end{lemma}

\begin{proof}
We use the following linear result \cite[Theorems~9.2 and~9.3(A)]{HarveyLawson}:
	if a locally integrable, lower
	semicontinuous function $v$ satisfies $\Delta v\leq q$ in the
	lower-test sense, with a smooth function $q$, then the same inequality holds
	in the distributional sense. 
	
	For $\kappa =0$, the statement of the lemma follows immediately.

For $\kappa>0$, the function $q:=\kappa e^{2u}$ is positive and
lower semicontinuous. Choose continuous $h_j:\Omega\to[0,\infty)$
with $h_j\uparrow q$ \cite[Theorem~2.1.3]{Ransford}.
Approximate the continuous functions  $h_j$ by smooth  $q_j$  with $h_j-2^{-j}\leq q_j\leq h_j$.
Then $q_j\to q$ pointwise and $-1\leq q_j\leq q$.
For every $x\in\Omega$ with $u(x)<\infty$, any lower test $\phi$
for $u$ at $x$ satisfies
\[
 \Delta\phi(x)+q_j(x)\leq0.
\]
Hence $\Delta u+q_j\leq0$  in the sense of distributions.
Dominated convergence implies $\Delta u+q\leq0$ in the
distributional sense.

Assume now that $\kappa<0$. For $t\in\R$,
set $u_t=\min\{u,t\}$. Each $u_t$ is lower semicontinuous and
locally bounded. It suffices to prove for every $t$ the distributional inequality:
\begin{equation} \label{eq:distr}
 \Delta u_t\leq -\kappa e^{2u_t} =|\kappa| e^{2u_t}\,.
\end{equation}
Indeed, as $n\to \infty$, the assumptions
$u,e^{2u}\in L^1_{\loc}(\Omega)$ give $u_n\to u$ and
$e^{2u_n}\to e^{2u}$ in $L^1_{\loc}$ by dominated convergence, which would imply $\Delta u\leq |\kappa|e^{2u}$.

If $\varphi$ is a lower test for $u_t$ at $x$ and $u_t(x)<t$, then
$\varphi$ is also a lower test for $u$ there. If $u_t(x)=t$, then
$\varphi\leq t$ near $x$ with equality at $x$, so
$\Delta\varphi(x)\leq0$. In both cases,
\begin{equation}\label{eq:distr2}
 \Delta\varphi(x)\leq |\kappa| e^{2u_t(x)}\leq |\kappa| e^{2t}.
\end{equation}
The linear implication therefore gives
$\Delta u_t\leq |\kappa|  e^{2t}$ in the distributional sense. Thus $\Delta u_t$
is a signed Radon measure, with Lebesgue decomposition
\[
 \Delta u_t=h_t\LL^2+\mu_t,\qquad
 \mu_t\perp\LL^2,\quad \mu_t\leq0,\quad h_t\leq |\kappa| e^{2t}
 \quad\text{a.e.}
\]

For $s<t$ the functions $u_t$ and $u_s$ both agree with $u$ on $A_s:=\{u <s\}$.  By \cite[Theorem~1.1, Equation~(1.3)]{AmbrosioPonceRodiac}, the densities $h_t$ and $h_s$ of the absolutely continuous parts of the Laplacians $\Delta u_t$  and $\Delta u_s$ agree almost everywhere on the subset $A_s$. Thus, 
\[
 h_t=h_s\leq |\kappa| e^{2s}
 \quad\text{almost everywhere on } A_s.
\]
For fixed $t$, we find a null set $\mathcal N$ such that for any $x\notin \mathcal N$ and any rational number  $s<t$,
\[
u(x)<s  \implies h_t(x) \leq |\kappa| e^{2s}\,.
\] 
Thus, choosing for each $x\in A_t \setminus \mathcal N$ a sequence of rational $s_n\downarrow u(x)$ we deduce
 \[
h_t(x)\leq |\kappa| e^{2u(x)}=|\kappa| e^{2u_t(x)}\,.
\]
On $\{u\geq t\}$, the same bound follows from $h_t\leq |\kappa| e^{2t}$,
since $u_t=t$ there. Together with $\mu_t\leq0$, this proves
\eqref{eq:distr} and finishes the proof of the lemma.
\end{proof}

It therefore remains to find a lower semicontinuous representative
of $W=\log f$ and to verify \eqref{eq:abstract-lower-test} for this representative.

\subsection{Radial averages and semicontinuity}
\label{subsec:direct-radial}

Fix a compact set $K\Subset D$ and choose $R,r_0>0$ as in
\cref{sec:radial}, with all estimates valid on a neighborhood of $K$.
For all $x\in K$, define
\[
 Q_{x,\theta}(t)=\int_0^t f(x+s\theta)\dd s,
 \qquad 0\leq t\leq R,
\]
allowing the value $+\infty$. Since $f\geq1$, we have
$Q_{x,\theta}(t)\geq t$. For $z=x+t\theta$ with $0<t<R$, write
$Q_x(z)=Q_{x,\theta}(t)$. For $0<r<r_0$ put
\begin{equation}\label{eq:H-definition}
\begin{aligned}
 H_x(r)
 &=\frac1{2\pi r^2}\int_{S^1}\int_0^r
         \min\{Q_{x,\theta}(t)^2,r^2\}\frac{\dd t}{t}\dd\theta
          -\frac12\\
 &=\frac1{2\pi r^2}\int_{B_r(x)}
       \frac{\min\{Q_x(z)^2,r^2\}-|z-x|^2}{|z-x|^2}\dd z.
\end{aligned}
\end{equation}
The second integrand is nonnegative, hence $H_x(r)\in[0,\infty]$.
If $f\equiv a\geq1$, then
\begin{equation}\label{eq:H-constant}
 H_x(r)=\frac1{r^2}\int_0^{r/a}a^2t\dd t
           +\int_{r/a}^r\frac{\dd t}{t}-\frac12=\log a.
\end{equation}

The following lemma shows that $H_x(r)$ approximates $W=\log f$
at almost every center.

\begin{lemma}[Semicontinuity and radial approximation]\label{lem:radial-H}
For every fixed $r\in(0,r_0)$, the function $x\mapsto H_x(r)$ is
lower semicontinuous on $K$. At almost every  $x\in K$,
\begin{equation}\label{eq:H-limit}
 H_x(r)\xrightarrow[r\downarrow0]{}W(x).
\end{equation}
\end{lemma}

\begin{proof}
For lower semicontinuity, fix $r\in(0,r_0)$. For $0<\eta<r$ and
$\eta<t<r$, set
\[
 Q_x^\eta(x+t\theta)=\eta+\int_\eta^t f(x+s\theta)\dd s,
\]
which corresponds to replacing $f$ by $1$ on $B_{\eta} (x)$. We further  define
\[
 H_x^\eta(r)=\frac1{2\pi r^2}
 \int_{B_r(x)\setminus B_\eta(x)}
 \frac{\min\{Q_x^\eta(z)^2,r^2\}-|z-x|^2}{|z-x|^2}\dd z.
\]
We claim that, for fixed $\eta$, the map $x\mapsto H_x^\eta(r)$ is continuous. This follows from the uniform boundedness of the integrands and the fact that, whenever $x_n\to x$, the functions $z\mapsto Q_{x_n}^{\eta}(x_n+z)$ converge to $z\mapsto Q_x^{\eta}(x+z)$ in $L^1(B_r(0)\setminus B_\eta(0))$.

Since $f\geq1$, we have $Q_x^\eta(x+t\theta)\geq t$ and
$Q_x^\eta(z)\uparrow Q_x(z)$ as $\eta\downarrow0$. The integrands,
extended by zero on $B_\eta(x)$, are nonnegative and increase as
$\eta\downarrow0$. Monotone convergence gives
$H_x^\eta(r)\uparrow H_x(r)$, proving lower semicontinuity
everywhere on $K$.

Now fix a good Lebesgue center $x$ and put $a=f(x)\geq1$.
By \eqref{eq:radial-diff},
\[
\hat  e_t(x):=\int_{S^1}\left|\frac{Q_{x,\theta}(t)}t-a\right|
          \dd\theta\xrightarrow[t\downarrow0]{}0.
\]
Comparing \eqref{eq:H-definition} with \eqref{eq:H-constant} and
using the $2r$-Lipschitz bound for the cutoff, we obtain
\[
\begin{aligned}
 |H_x(r)-\log a|
 &\leq\frac1{2\pi r^2}\int_0^r\int_{S^1}
       \left|\min\{Q_{x,\theta}(t)^2,r^2\}
              -\min\{a^2t^2,r^2\}\right|
          \dd\theta\frac{\dd t}{t}\\
 &\leq\frac1{\pi r}\int_0^r\int_{S^1}
       \left|\frac{Q_{x,\theta}(t)}t-a\right|\dd\theta\dd t\\
 &=\frac1{\pi r}\int_0^r \hat e_t(x)\dd t
   \leq\frac1\pi\sup_{0<t<r} \hat e_t(x)
   \xrightarrow[r\downarrow0]{}0.
\end{aligned}
\]
This proves \eqref{eq:H-limit}.
\end{proof}

To obtain a lower semicontinuous representative of $W$ from the
lower semicontinuity of $x\mapsto H_x(r)$, we show that the convergence
in \eqref{eq:H-limit} is from below after adding the following
$\kappa$-dependent correction:
\begin{equation}\label{eq:radial-correction}
 \Vk(r)=\int_0^r\frac{\pi s^2-\vk(s)}{\pi s^3}\dd s
       =\frac{\kappa r^2}{24}+O(r^4).
\end{equation}
The expansion follows from
$\vk(r)=\pi r^2-\pi\kappa r^4/12+O(r^6)$.

\begin{lemma}\label{lem:radial-comparison}
At almost every center $x\in K$, the function $H_x+\Vk$ is locally
Lipschitz and nonincreasing on $(0,r_0)$. In particular,
\begin{equation}\label{eq:H-comparison}
 H_x(r)+\Vk(r)\leq W(x),\qquad0<r<r_0.
\end{equation}
\end{lemma}

\begin{proof}
Fix a good Lebesgue center $x$. By \eqref{eq:radial-diff} and
integrability away from $x$, there is $C_x>0$ such that
\[
 \int_{S^1}Q_{x,\theta}(t)\dd\theta\leq C_x t,
 \qquad0<t<r_0.
\]
This estimate justifies differentiation under the
integral in \eqref{eq:H-definition} and shows that $H_x$, and hence
$H_x+\Vk$, is locally Lipschitz.

Differentiating the first formula in
\eqref{eq:H-definition}, we obtain, for almost every $r\in(0,r_0)$,
\begin{align}
 H_x'(r)
 &=\frac{\dd}{\dd r}\left[
   \frac1{2\pi}\int_{S^1}\int_0^r
   \min\left\{\frac{Q_{x,\theta}(t)^2}{r^2},1\right\}
       \frac{\dd t}{t}\dd\theta\right]\notag\\
 &=\frac1r+\frac1{2\pi}\int_{S^1}\int_0^r
     \frac{\partial}{\partial r}
     \min\left\{\frac{Q_{x,\theta}(t)^2}{r^2},1\right\}
       \frac{\dd t}{t}\dd\theta\notag\\
 &=\frac1r-\frac1{\pi r^3}\int_{S^1}\int_0^r
     Q_{x,\theta}(t)^2\mathbf1_{\{Q_{x,\theta}(t)<r\}}
       \frac{\dd t}{t}\dd\theta\notag\\
 &=\frac1r-\frac{J_x(r)}{\pi r^3}.
 \label{eq:H-derivative}
\end{align}
Here $J_x(r)$ denotes the double integral in the third line.
The boundary term is $1/r$ because $Q_{x,\theta}(r)\geq r$, and hence $\min\left\{\frac{Q_{x,\theta}(r)^2}{r^2},1\right\}=1$.
The interior derivative is $-2Q_{x,\theta}(t)^2/r^3$ where
$Q_{x,\theta}(t)<r$, and zero where $Q_{x,\theta}(t)>r$.

We now use the same square inequality as in the proof of
\cref{prop:positivity}:
\[
 t f(x+t\theta)^2+\frac{Q_{x,\theta}(t)^2}{t}
       \geq2Q_{x,\theta}(t)f(x+t\theta).
\]
Integrate over $\{0<t<r:Q_{x,\theta}(t)<r\}$ and over $S^1$.
The first term is $A_x(r)$ and the second term is $J_x(r)$. Along each good ray, absolute continuity,
$Q_{x,\theta}'=f(x+t\theta)$, and $Q_{x,\theta}(r)\geq r$ give
\[
 2\int_{\substack{0<t<r\\Q_{x,\theta}(t)<r}}
             Q_{x,\theta}(t)f(x+t\theta)\dd t=r^2.
\]
Thus $A_x(r)+J_x(r)\geq2\pi r^2$. Combining this with
\eqref{eq:H-derivative} and \eqref{eq:A-upper}, we obtain
\[
 H_x'(r)\leq   \frac 1 r  - \frac{2\pi r^2- A_x(r)}{\pi r^3} =
  \frac{ A_x(r)-\pi r^2}{\pi r^3} 
          \leq\frac{\vk(r)-\pi r^2}{\pi r^3}=-\Vk'(r)\,,
          \]
          for almost every $r$.
Local Lipschitz continuity and the last inequality imply that
$H_x+\Vk$ is nonincreasing. Together with \eqref{eq:H-limit} and $\Vk(r)\to0$ as $r\downarrow0$,
this proves \eqref{eq:H-comparison}.
\end{proof}

We now choose the representative needed for pointwise testing.

\begin{lemma}[A representative for pointwise tests]\label{lem:representative}
The function
\begin{equation}\label{eq:W-definition}
 W_*(x)=\sup_{0<r<r_0}\{H_x(r)+\Vk(r)\}
\end{equation}
is nonnegative, lower semicontinuous, and equal to $W$ almost
everywhere. It satisfies
\begin{equation}\label{eq:essential-representative}
 W_*(x)=\essliminf_{z\to x}\log f(z).
\end{equation}
\end{lemma}

\begin{proof}
The first assertions follow from
\cref{lem:radial-H,lem:radial-comparison},
$H_x(r)\geq0$, and $\Vk(r)\to0$. Lower semicontinuity and equality
almost everywhere give
$W_*(x)\leq\essliminf_{z\to x}\log f(z)$.
Conversely, suppose $f\geq a\geq1$ almost everywhere near $x$.
Polar integration about this fixed center gives
$Q_{x,\theta}(t)\geq at$ for almost every direction and all small
$t$. Comparison with the constant factor $a$ in
\eqref{eq:H-definition} yields $H_x(r)\geq\log a$ for small $r$.
Letting $r\downarrow0$ gives $W_*(x)\geq\log a$. This proves \eqref{eq:essential-representative}.
\end{proof}

\subsection{Lower tests and curvature}

From now on, we use this representative of $W$. It is nonnegative,
belongs to $L^1(D)$, and may be $+\infty$ on a null set.
At every center and every admissible radius, its definition gives
\begin{equation}\label{eq:allcenter-comparison}
 H_x(r)+\Vk(r)\leq W(x).
\end{equation}
This comparison is needed at every point $x$, not only at good centers.

\begin{lemma}[The sharp lower-test inequality]\label{lem:viscosity}
	If $\varphi$ is  a lower test for $W$ at a point $x$ where
$W(x)<\infty$, then
\begin{equation}\label{eq:viscosity}
 \Delta\varphi(x)+\kappa e^{2\varphi(x)}\leq0.
\end{equation}
\end{lemma}

\begin{proof}
The idea is the following. A violation of \eqref{eq:viscosity} means
that the $C^2$ Riemannian metric $e^{2\varphi}|\dd z|^2$ has curvature
strictly less than $\kappa$ at $x$. Consequently, sufficiently small
balls centered at $x$ in this metric have area greater than $\vk(r)$.
This suggests seeking the opposite comparison for the radial averages
associated with $e^\varphi$ by a Taylor expansion. Such a comparison would
contradict the lower-test condition and \eqref{eq:allcenter-comparison}.
This argument works once $\nabla\varphi(x)=0$. We therefore begin with a
conformal change of coordinates.

We first record the effect of a conformal change of coordinates.
Let $F:V\to\widetilde V$, with $V\subset D$, be a conformal
diffeomorphism between planar domains, and put $h=\log|F'|$.
In the coordinates $w=F(z)$, the parametrization becomes
$\widetilde\Phi=\Phi\circ F^{-1}$, with conformal factor and logarithm
\begin{equation}\label{eq:coordinate-factor}
 \widetilde f\circ F=e^{-h}f,\qquad
 \widetilde W\circ F=W-h.
\end{equation}
The second identity holds pointwise for the representatives defined
by \eqref{eq:essential-representative}, since $F$ preserves null sets
in both directions and $h$ is continuous. The length formula, the
one-sided area formula, and local $L^2$ integrability are invariant
under this change; the target metric and $\kappa$ remain unchanged.
If $\varphi$ touches $W$ from below at $x\in V$, then
\[
 \widetilde\varphi=(\varphi-h)\circ F^{-1}
\]
touches $\widetilde W$ from below at $F(x)$. Since $h$ is harmonic,
the transformation rule for the planar Laplacian gives
\[
 \bigl(\Delta_w\widetilde\varphi+\kappa e^{2\widetilde\varphi}\bigr)(F(x))
 =|F'(x)|^{-2}
  \bigl(\Delta_z\varphi+\kappa e^{2\varphi}\bigr)(x).
\]
Thus \eqref{eq:viscosity} is invariant under conformal coordinate changes.

We now choose conformal coordinates in which the lower test has
zero gradient, while retaining $\widetilde f\geq1$ and arranging
$\widetilde\varphi>0$. Choose a harmonic affine function $h$ with
$\nabla h(x)=\nabla\varphi(x)$. After subtracting a constant from $h$ and
restricting the neighborhood, we may also assume $h\leq0$ and
$\varphi-h>0$ there. Let $\widehat h$ be a harmonic conjugate of $h$.
A holomorphic primitive $F$ with $F'=\exp(h+i\widehat h)$ is an
injective coordinate change after restriction. The transformed
factor satisfies $\widetilde f\geq1$, and the transformed lower test
satisfies $\widetilde\varphi>0$ and $\nabla_w\widetilde\varphi(F(x))=0$.

Applying \cref{subsec:direct-radial} to the transformed chart, we
obtain the comparison \eqref{eq:allcenter-comparison} in these coordinates as well.

We may therefore assume $x=0$, $\nabla\varphi(0)=0$, and $\varphi>0$
near zero. Put $a=e^{\varphi(0)}>1$ and $g=e^\varphi$.
We now define $Q_{g,0,\theta}$ and $H_{g,0}$ by replacing $f$ with $g$
in the radial constructions, and estimate them using the Taylor
expansion of $\varphi$.

Uniformly in $\theta\in S^1$,
\[
 Q_{g,0,\theta}(t):=\int_0^t g(s\theta)\dd s
   =a\left(t+\frac{\varphi_{\theta\theta}(0)}6t^3+o(t^3)\right).
\]
Write $Q_\theta=Q_{g,0,\theta}$. Squaring gives, again uniformly in
$\theta$,
\[
 Q_\theta(t)^2-a^2t^2
 =\frac{a^2}{3}\varphi_{\theta\theta}(0)t^4+o(t^4).
\]
Using \eqref{eq:H-definition} for the $C^2$ factor $g$, we write
\[
 H_{g,0}(r)+\frac12
 =\frac1{2\pi r^2}\int_{S^1}\int_0^r
       \min\{Q_\theta(t)^2,r^2\}\frac{\dd t}{t}\dd\theta.
\]

Since $Q_\theta(t)=at+O(t^3)$, the transition $Q_\theta(t)=r$
occurs at $t=r/a+O(r^3)$, uniformly in $\theta$. Thus
\[
 \min\{Q_\theta(t)^2,r^2\}-\min\{a^2t^2,r^2\}
\]
agrees with $(Q_\theta(t)^2-a^2t^2)\mathbf1_{\{t<r/a\}}$
outside an interval $|t-r/a|\leq C_0r^3$, for some $C_0>0$
independent of $\theta$. Their difference is $O(r^4)$ on that
interval. Since $t$ is comparable to $r$ there, integrating the error
against $r^{-2}\dd t/t$ contributes $O(r^4)$, uniformly in $\theta$.
Since $H_{a,0}(r)=\log a$ by \eqref{eq:H-constant}, we obtain
\begin{align}
 H_{g,0}(r)-\log a
 &=\frac1{2\pi r^2}\int_{S^1}\int_0^{r/a}
       \bigl(Q_\theta(t)^2-a^2t^2\bigr)\frac{\dd t}{t}\dd\theta
       +o(r^2)\notag\\
 &=\frac{a^2}{6\pi r^2}\int_{S^1}\varphi_{\theta\theta}(0)
                   \int_0^{r/a}t^3\dd t\dd\theta+o(r^2)\notag\\
 &=\frac{e^{-2\varphi(0)}\Delta\varphi(0)}{24}r^2+o(r^2).
 \label{eq:smooth-H-expansion}
\end{align}
Since $f\geq g$ almost everywhere near zero, polar integration at
this fixed center gives $Q_{f,0,\theta}\geq Q_{g,0,\theta}$ for
almost every direction and all sufficiently small $t$. The definition of $H$ is monotone in $Q$,
so \eqref{eq:allcenter-comparison} yields, for all small $r$,
\[
 H_{g,0}(r)+\Vk(r)\leq H_{f,0}(r)+\Vk(r)
                   \leq W(0)=\varphi(0).
\]
Insert \eqref{eq:radial-correction} and
\eqref{eq:smooth-H-expansion}, divide by $r^2/24$, and let
$r\downarrow0$. This proves \eqref{eq:viscosity}.
\end{proof}

\begin{proof}[Proof of \cref{prop:analytic}]
The representative in \cref{lem:representative} satisfies the
hypotheses of \cref{lem:distribution}, by \cref{lem:viscosity} and
the integrability established at the end of \cref{sec:radial}.
Thus \eqref{eq:target} holds for this representative and hence for
$W=\log f$. 
\end{proof}

\section{Identification of the conformal metric}
\label{sec:identification}
Let $\Phi:D\to U$ be a chart supplied by \cref{prop:charts}, and
let $W$ be the representative of $\log f$ constructed in
\cref{lem:representative}. Thus $W$ is lower semicontinuous and locally
bounded below, $\mu:= -\Delta W$ is a signed Radon measure, and
\begin{equation}\label{eq:curvature-measure}
	\mu=-\Delta W\geq\kappa e^{2W}\dd z.
\end{equation}
We first investigate the canonical conformal metric associated with $W$ away
from its possible cusp points. We will freely use the theory of
surfaces of bounded integral curvature, as presented in \cite{Reshetnyak}.

We call a point $z\in D$ a \emph{cusp point} if
$\mu(\{z\})\geq2\pi$. We record the following observation.

\begin{lemma}\label{lem:cusps}
The set $S$ of cusp points is discrete in $D$. If $\kappa\geq0$,
then $S$ is empty.
\end{lemma}

\begin{proof}
Since $|\mu|$ is finite on compact subsets and each cusp point has
mass at least $2\pi$, every compact subset of $D$ contains only
finitely many cusp points.

Suppose now that $\kappa\geq0$ and $b=\mu(\{z\})\geq2\pi$.
Then $\Delta W\leq0$, so $W$ is superharmonic. By the Riesz
decomposition theorem \cite[Theorem~3.7.9]{Ransford}, we obtain
\[
 W(x)\geq\frac b{2\pi}\log\frac1{|x-z|}-C
      \geq\log\frac1{|x-z|}-C
\]
for some $C\in\R$ and almost every $x$ in a
neighborhood of $z$. Hence
$e^{2W(x)}\geq e^{-2C}|x-z|^{-2}$ almost everywhere near $z$,
contradicting the local integrability of $e^{2W}$.
\end{proof}

For the remainder of this section, assume that $D$ contains no cusp
points. Since $W$ is lower semicontinuous, the integral
$\int_\gamma e^W\dd s\in[0,\infty]$ is well-defined for every
Euclidean rectifiable curve $\gamma$ in $D$. The associated
conformal distance is given by
\begin{equation}\label{eq:conformal-distance}
 \tilde d(x,y)=\inf_\gamma\int_\gamma e^W\dd s,
\end{equation}
where the infimum is taken over all Euclidean rectifiable curves in
$D$ joining $x$ to $y$, and $\dd s$ denotes Euclidean arclength.
By Reshetnyak's theory and the absence of cusp points, $\tilde d$
is a finite length metric inducing the Euclidean topology on $D$;
see \cite[Theorem~7.1.1]{Reshetnyak} and \cite[Section~6.2]{LS}.
We denote the resulting metric space by $Y=(D,\tilde d)$. Its
Hausdorff area measure is $e^{2W}\dd z$, and its so-called curvature measure
is $\mu$; see \cite[Chapter~7]{Reshetnyak}.

Reshetnyak's theory and the stability of lower curvature bounds under
Gromov--Hausdorff convergence yield the following result, as explained
in the proofs of \cite[Corollary~5.2 and Lemma~6.2]{LS}.
See also \cite[Section~3, proof of Theorem~1]{Marot} for an independent proof of this statement.

\begin{lemma}[Lower curvature bound]\label{lem:canonical}
The space $Y$ is locally $\CBB(\kappa)$.
\end{lemma}

For \emph{every} Euclidean rectifiable curve $\gamma$ in $D$, its
length in $Y$ equals $\int_\gamma e^W\dd s$, by
\cite[Theorem~5]{ReshetnyakLength}. For \emph{$2$-almost every} such
curve, this integral equals $\length_X(\Phi\circ\gamma)$ by
\cref{prop:charts}\textup{(ii)}. Together with lower
semicontinuity of length in $X$, these imply the following result,
as in \cite[proof of Lemma~9.3]{LWcharacterization}; see also the
forward-map argument in \cite[proof of Lemma~6.3]{LS}.

\begin{lemma}\label{lem:short}
The homeomorphism $\Phi:Y\to U\subset X$ is locally $1$-Lipschitz.
\end{lemma}

In particular, $\Phi$ does not increase the Hausdorff area of any
Borel subset $E\subset Y$. Combining this with the area formula
\eqref{eq:E} and the identity
$\HH_Y^2(E)=\int_E e^{2W}\dd z=\int_E f^2\dd z$ gives the following
result.
\begin{lemma}\label{lem:area}
The map $\Phi:(Y,\HH_Y^2)\to(U,\HH_U^2)$ is area-preserving.
\end{lemma}

To show that $\Phi$ does not contract distances locally, we invoke
once more the upper bounds for the areas of balls in $X$.

\begin{proposition}\label{prop:area-rigidity}
The map $\Phi$ is a local isometry. Hence $U$ is locally
$\CBB(\kappa)$.
\end{proposition}

\begin{proof}
The statement is local. Fix a point of $Y$ and choose a relatively
compact geodesic open neighborhood $Z\subset Y$ whose closure,
with the restricted metric, is a compact Alexandrov space; see
\cite[Proposition~1.1 and Lemma~3.8]{LebedevaNepechiy}.
We may assume in addition that  the restriction of $\Phi$ to $Z$  is $1$-Lipschitz.

We claim that for every $\eps>0$ there is a finite set
$F_\eps\subset Z$ with the following property. If $x_j\to x\in Z$
and $\delta_j\downarrow0$ satisfy
\[
 \HH_Y^2(B_{\delta_j}^Y(x_j))
       \leq(\pi-\eps)\delta_j^2,
\]
then $x\in F_\eps$.

Indeed, the Bishop--Gromov inequality
\cite[Corollary~10.13]{BGP}, applied at $x_j$ and followed by
$j\to\infty$, gives
\[
 \HH_Y^2(B_r^Y(x))
 \leq\left(1-\frac{\eps}{\pi}\right)\vk(r)
\]
for all sufficiently small $r>0$. By
\cite[Lemma~8.1.1]{Reshetnyak} and the coarea formula,
\[
 \lim_{r\downarrow0}\frac{\HH_Y^2(B_r^Y(x))}{r^2}
 =\pi-\frac12\mu(\{x\}).
\]
Thus $\mu(\{x\})\geq2\eps$, so we may take
$F_\eps=\{q\in Z:\mu(\{q\})\geq2\eps\}$.
This set is finite since $|\mu|(\overline Z)<\infty$.

Fix $\lambda>1$, set $\eta=(\lambda-1)/(\lambda+1)$, and choose
$\eps>0$ so small that
\[
 (\pi-\eps)(1+\eta^2)>\pi.
\]
We show that $\Phi^{-1}$ is locally $\lambda$-Lipschitz on
$\Phi(Z\setminus F_\eps)$. Otherwise, there are
$p_j,q_j\to x\in Z\setminus F_\eps$ such that
\[
 L_j:=d_Y(p_j,q_j)>\lambda d_j,
 \qquad d_j:=d_X(\Phi(p_j),\Phi(q_j)).
\]
Put
\[
 R_j=\frac{L_j+d_j}{2},\qquad
 \rho_j=\frac{L_j-d_j}{2}>\eta R_j.
\]
The balls $B_{R_j}^Y(p_j)$ and $B_{\rho_j}^Y(q_j)$ are disjoint
and, for large $j$, contained in $Z$. Since
$d_j+\rho_j=R_j$, their images both lie in
$B_{R_j}^X(\Phi(p_j))$. By the claim, their areas are greater
than $(\pi-\eps)R_j^2$ and $(\pi-\eps)\rho_j^2$, respectively,
for all large $j$. Area preservation and \eqref{eq:main-ball}
therefore imply
\[
 (\pi-\eps)(1+\eta^2)R_j^2
 <\HH_X^2(B_{R_j}^X(\Phi(p_j)))
 \leq\vk(R_j)=\pi R_j^2+O(R_j^4).
\]
Dividing by $R_j^2$ and letting $j\to\infty$ gives a
contradiction.

Since $F_\eps$ is finite, we deduce  for 
every rectifiable curve $\beta$ in $\Phi(Z)$:
\[
 \length_Y(\Phi^{-1}\circ\beta)\leq\lambda\length_X(\beta).
\]
Letting $\lambda\downarrow1$ and using the local shortness of $\Phi$,
we conclude that $\Phi|_Z$ preserves all curve lengths.
Since $X$ and $Y$ are locally length spaces, $\Phi:Y\to X$ is a local isometry.
\end{proof}

\section{Proof of the main theorem}
\label{sec:cusps}
\label{sec:conclusion}

\begin{proof}[Proof of \cref{thm:main}]
Suppose first that $X$ has curvature locally bounded below by
$\kappa$. Alexandrov spaces satisfy property~(ET); see
\cite[Section~11]{LWArea}. This property is local, so it also holds
for $X$. The Bishop--Gromov inequality gives the ball-area bound
\eqref{eq:main-ball}; see
\cite[Theorem~10.2]{BGP} and \cite[Lemma~4.4]{LebedevaNepechiy}.

Conversely, suppose that $X$ satisfies property~(ET) and
\eqref{eq:main-ball}. Fix $x\in X$ and choose a chart
$\Phi:D\to U$ about $x$ as in
\cref{prop:charts}. \Cref{prop:positivity} gives the positive local
lower bound for $f$, and \cref{lem:representative,prop:analytic} give
the representative $W$ satisfying \eqref{eq:curvature-measure}.
Let $S\subset D$ be its set of cusp points. By \cref{lem:cusps},
$S$ is locally finite and is empty if $\kappa\geq0$. Applying
\cref{prop:area-rigidity} on coordinate discs compactly contained
in $D\setminus S$, we see that $X$ is locally $\CBB(\kappa)$ on
$U\setminus\Phi(S)$, and that its metric agrees locally there
with the canonical conformal metric.

Thus it remains to consider the case $z_*:=\Phi^{-1}(x)\in S$.
Then $\kappa<0$. After restricting $\Phi$ to a smaller 
disc about $z_*$ and replacing $U$ by its image, we may assume
that $S=\{z_*\}$.

Define $\tilde d$ on all of $D$ by \eqref{eq:conformal-distance}.
The local identification away from $z_*$ gives
$\length_X(\Phi\circ\gamma)=\int_\gamma e^W\dd s$ for every
simple Euclidean rectifiable arc $\gamma\subset D$: if the arc
meets $z_*$, split it there and pass to the limit on its subarcs.
Conversely, the lift of any simple rectifiable arc in $U$ is
Euclidean rectifiable, since $e^W$ is locally bounded below.
Taking infima over these arcs identifies $\tilde d$ with the
intrinsic metric of $U$. Thus $\tilde d$ is finite and, by
\cite[Theorem~7.1.1]{Reshetnyak}, is a metric of bounded integral
curvature inducing the coordinate topology. In particular, $U$
is locally a surface of bounded integral curvature also at $x$.

By \cite[Lemma~8.1.1, p.~124]{Reshetnyak}, for all sufficiently
small $r>0$ the metric circle
$C_r=\{y\in U:d_X(x,y)=r\}$ is a simple closed curve and
\[
 0\leq\lim_{r\downarrow0}\frac{\length_X(C_r)}{r}
 =2\pi-\mu(\{z_*\})\leq0.
\]
The point $x$ cannot be a \emph{transit point}, no unit-speed minimizing geodesic $\gamma:[-a,a]\to U$ with
$\gamma(0)=x$ exists. Otherwise, such a geodesic would satisfy, for small $r>0$,
\[
 2r=d_X(\gamma(-r),\gamma(r))
 \leq\tfrac12\length_X(C_r)=o(r),
\]
by joining its two intersections with $C_r$ along the shorter
arc of that circle. This is a contradiction.

By \cite[Proposition~5.3]{CreutzRomney}, there is a compact
absolutely convex polygon $P\subset U$ containing $x$, whose
boundary consists of transit points of $X$. Therefore $x\in P^\circ$.
Equip $P$ and $P\setminus\{x\}$ with the restricted metric of $X$.
Convexity supplies minimizing geodesics in $P$ between any two of
its points. If neither endpoint is $x$, such a geodesic cannot
pass through $x$, by the preceding argument. Hence
$P\setminus\{x\}$ is a geodesic space. It is locally
$\CBB(\kappa)$, including at its boundary: the local four-point
comparison inequalities are inherited from $U\setminus\{x\}$.
By Petrunin's globalization theorem \cite{PetruninGlobal}, its
completion, which is $P$, is $\CBB(\kappa)$. Since $x\in P^\circ$,
$X$ is locally $\CBB(\kappa)$ at $x$ as well. This completes the
proof.
\end{proof}

\appendix
\section{Passing to the intrinsic metric}\label{app:intrinsic}

Write $d_{\mathrm{length}}$ for the extended length metric induced by
$d$, allowing the value $+\infty$.

\begin{proposition} \label{prop:app}
Let $(X,d)$ be a connected metric surface satisfying property~\textup{(ET)}
and \eqref{eq:main-ball}.
Then there exists a discrete, possibly empty, subset $S\subset X$ with the
following property. No point of $S$ can be joined by a curve of finite
length to any other point of $X$, and $X\setminus S$ is connected and
locally $\CBB(\kappa)$ with respect to the induced length metric
$d_{\mathrm{length}}$.

If $\kappa\geq 0$ then $S$ is empty.
\end{proposition}

\begin{proof}
Fix $x\in X$. The ball-area bound gives a quadratic upper area bound on
compact subsets. Hence the local form of \cite[Theorem~1.6]{Rajala},
together with \cite[Theorem~1.4]{Rajala}, provides a neighborhood
$U$ of $x$ and a geometrically quasiconformal homeomorphism
\[
 \Phi:D\longrightarrow U .
\]
By the analytic characterization of quasiconformality,
\cite[Theorem~6.18]{NtalampekosSurvey}, we have
$\Phi\in W^{1,2}_{\loc}(D,X)$. Property~\textup{(ET)} implies that
its approximate metric differentials form a measurable field of
inner-product seminorms. These seminorms are nondegenerate almost
everywhere by \cite[Lemma~14.2]{Rajala}, and quasiconformality gives
a uniform bound on their eccentricity. Applying the measurable
Riemann mapping theorem as in \cite[Section~14]{Rajala} and changing
coordinates, we may therefore assume that
\[
 \apmd\Phi_z=f(z)|\cdot|
 \qquad\text{for almost every }z\in D,
\]
where $f\in L^2_{\loc}(D)$. After restricting
$D$, the length identity for $2$-almost every curve and the one-sided
area formula in \cref{prop:charts} hold exactly as before. Moreover,
$\Phi^{-1}$ has Lusin's property~\textup{(N)} by
\cite[Theorem~6.20]{NtalampekosSurvey}. The final argument in the proof
of \cref{prop:charts} therefore shows that $f>0$ almost everywhere.

The arguments of \cref{sec:radial,sec:representative} use only these
properties of $\Phi$, the ball-area bound, and injectivity, and hence
apply without change. Thus \cref{prop:positivity} gives
$f^{-1}\in L^\infty_{\loc}(D)$, and we obtain the representative
$W=\log f$ satisfying \eqref{eq:curvature-measure}.

Set $z=\Phi^{-1}(x)$. By \cref{lem:cusps}, after restricting $D$
we may assume that the set of cusps is empty or consists only of $z$.
Moreover, no cusps occur for $\kappa\geq0$.

As before, we define  the conformal metric $\tilde d$ by 
\eqref{eq:conformal-distance}.  However, if $z$  is a cusp, the $\tilde d$-distance of any other point in $D$  to $z$ might be infinite.
On the complement of $z$ the finite metric  $\tilde d$  is locally $\CBB(\kappa)$. 
Thus, the resulting extended metric space $Y=(D,\tilde d)$ has either two finite-distance components, $D\setminus\{z\}$ and $\{z\}$, or a single finite-distance component.  In the second case, we see in the same way as at the end of Section \ref{sec:conclusion}, that all of $Y$ is locally $\CBB(\kappa)$.  (Note that in this case $z$ cannot be a cusp).

In both cases, the same argument as in Section \ref{sec:identification} shows that the map $\Phi:Y\to U$ is a bijection, which does not increase length of rectifiable curves and preserves the area.  The same argument as in the proof of Proposition \ref{prop:area-rigidity} shows that $\Phi^{-1}$  does not increase length of curves as well. 

We conclude that $Y$ and $(U,d_{\mathrm{length}})$ are locally isometric $\CBB(\kappa)$ spaces if $z$ is not a cusp and that $U\setminus \{x\}$ is a locally $\CBB(\kappa)$ space with respect to its length metric, if $z$ is a cusp.  In the second case, length preservation shows that $x$ is at infinite $d_{\mathrm{length}}$-distance from any point of $X\setminus\{x\} $.

We finish the proof defining $S\subset X$ to be the set of points which are not connected to any other point of $X$ by a curve of finite length.  The above considerations show that $S$ is discrete, that $S$ is empty if $\kappa \geq 0$ and that $X\setminus S$ is locally $\CBB(\kappa)$ with respect to the induced length metric. 
\end{proof}

For negative curvature the induced length metric need not be finite.
Consider the hyperbolic cusp
\[
 C=(0,\infty)\times(\R/\mathbb Z),\qquad
 ds^2=dt^2+e^{-2t}d\theta^2,
\]
with intrinsic distance $d_C$, and set $h(t,\theta)=e^{-t/2}$.
On $X=C\cup\{p\}$ define
\[
 d(p,x)=h(x),\qquad
 d(x,y)=\min\{d_C(x,y),h(x)+h(y)\}.
\]
It is not  difficult to see that the topological disc $X$ satisfies the assumptions of 
Proposition \ref{prop:app}  while the induced length metric on $X$ consists of two components $C$ and $\{p\}$.

\end{document}